\documentclass[11pt]{article}
\usepackage[utf8]{inputenc}
\usepackage{amsmath}
\usepackage{amsthm}
\usepackage{amssymb, color}
\usepackage{thmtools}
\usepackage{thm-restate}

\newcommand{\PR}{\mathbb{P}}
\newcommand{\C}{\mathcal{C}}
\newcommand{\D}{\Delta}

\newcommand{\e}{\varepsilon}
\newcommand{\De}{\D_{\e}}
\newcommand{\de}{d_{\e}}
\newcommand{\DLG}{\D_G^L}
\newcommand{\E}{\mathbb{E}}
\newcommand{\LL}{Lov\'{a}sz Local Lemma}
\newcommand{\Lceil}{\left\lceil}
\newcommand{\Rceil}{\right\rceil}

\newtheorem*{LAC}{The Linear Arboricity Conjecture (LAC)}
\newtheorem*{LAC'}{The Linear Arboricity Conjecture' (LAC')}
\newtheorem*{LLAC}{The List Linear Arboricity Conjecture (LLAC)}
\newtheorem*{LLL}{Lov\'{a}sz Local Lemma}
\newtheorem*{WLLL}{General Version of the Lov\'{a}sz Local Lemma}
\newtheorem*{CB}{Chernoff's Bound}
\newtheorem*{TI}{Talagrand's Inequality}
\newtheorem*{LCC}{The List Coloring Conjecture}
\newtheorem*{KAHN}{Kahn's Theorem}

\newtheorem{THM}{Theorem}[section]
\newtheorem{LEM}[THM]{Lemma}
\newtheorem{COR}[THM]{Corollary}
\newtheorem{PROP}[THM]{Proposition}

\begin{document}
\font\smallrm=cmr8

\baselineskip=12pt
\phantom{a}\vskip .25in
\centerline{{\bf  The List Linear Arboricity of Graphs}}
\vskip .3in
\centerline{{\bf Ringi Kim}%
\footnote{\texttt{rin@kaist.ac.kr}.}}
\smallskip
\centerline{Department of Mathematical Sciences}
\centerline{KAIST}
\centerline{Daejeon}
\centerline{South Korea 34141}
\bigskip
\centerline{and}
\bigskip
\centerline{{\bf Luke Postle}%
\footnote{\texttt{lpostle@uwaterloo.ca.} This research was partially
supported by NSERC under Discovery Grant No. 2014-06162, the Ontario Early 
Researcher Awards program and the Canada Research Chairs program.}}
\smallskip
\centerline{Department of Combinatorics and Optimization}
\centerline{University of Waterloo}
\centerline{Waterloo, ON}
\centerline{Canada N2L 3G1}

\vskip 0.4in \centerline{\bf ABSTRACT}
\bigskip

{
\parshape=1.0truein 5.5truein
\noindent

A \emph{linear forest} is a forest in which every connected component is a path. The \emph{linear arboricity} of a graph $G$ is the minimum number of linear forests of $G$ covering all edges. In 1980, Akiyama, Exoo and Harary proposed a conjecture, known as the Linear Arboricity Conjecture (LAC), stating that every $\D$-regular graph $G$ has linear arboricity $\Lceil \frac{\D+1}{2}  \Rceil$. In 1988, Alon proved that the LAC holds asymptotically. In 1999, the list version of the LAC was raised by An and Wu, which is called the List Linear Arboricity Conjecture. In this article, we prove that the List Linear Arboricity Conjecture holds asymptotically.

}

\vfill \baselineskip 11pt \noindent November 23, 2017.
\vfil\eject
\baselineskip 18pt

\section{Introduction}
In this paper, we consider only undirected simple graphs. 
A \emph{linear forest} is a forest in which every connected component is a path. The \emph{linear arboricity} of a graph $G$, denoted by $la(G)$, first introduced by Harary~\cite{HARARY}, is the minimum number of linear forests of $G$ needed to cover all edges of $G$. 
Akiyama, Exoo and Harary~\cite{AEH} proposed a conjecture, known as the Linear Arboricity Conjecture, stating that for every $\D$-regular graph $G$, $la(G) = \Lceil \frac{\D+1}{2}\Rceil$. It is easy to see as shown in~\cite{ALON,list-planar} that this conjecture is equivalent to the following:

\begin{LAC}
For every graph $G$ with maximum degree $\D$, $\Lceil \frac{\D}{2} \Rceil \le la(G) \le \Lceil \frac{\D+1}{2}\Rceil$.
\end{LAC}

The lower bound is easily obtained since at least $\Lceil \frac{\D}{2} \Rceil$ linear forests are needed to cover all edges incident with a vertex with degree $\D$. However, despite much effort, the conjecture for the upper bound is still open.
It has been proved only for several special cases: complete graphs~\cite{SCJ}, complete bipartite graphs~\cite{AEH}, series parallel graphs~\cite{SP} and planar graphs~\cite{planar1,planar2}. It is also proved that the LAC is true when $\D=3,4,5,6,8,10$ (see ~\cite{AEH,AEH4,EP, GF}).
For general graphs $G$, the best known upper bound of $la(G)$, due to Guldan~\cite{GF2}, is $\Lceil\frac{3\D}{5}\Rceil$ for even $\D$ and $\Lceil \frac{3\D+2}{5}\Rceil$ for odd $\D$. 
Alon~\cite{ALON,PM} proved that the LAC holds asymptotically as follows.

\begin{THM}[Alon~\cite{ALON, PM}]\label{THM:alon}
For every $\epsilon >0$, there exists $\De$ such that for every $\D>\De$, every $\D$-regular graph has linear arboricity at most $\frac{\D}{2}(1+\e)$.
0\end{THM}

A covering by linear forests can be viewed as an edge coloring where each color class induces a linear forest. 
Thus viewed as a coloring problem,  it is natural to consider its list version as follows. 

First, let us recall the definition of list chromatic index for comparison. For a graph $G$ and a list assignment $L=(L(e) : e \in E(G))$ to edges of $G$, the \emph{size of $L$}, which we denote by $|L|$, is the minimum of $|L(e)|$ taken over all $e\in E(G)$. An \emph{$L$-coloring} is a map $\phi$ defined on $E(G)$ such that $\phi(e)\in L(e)$ for every $e\in E(G)$, and for every color $c \in \bigcup_{e\in E(G)} L(e)$, $\phi^{-1}(c)$ induces a matching. The \emph{list chromatic index}, denoted $ch'(G)$ is the minimum $k$ such that for every list assignment $L=(L(e): e \in E(G))$  with $|L|\ge k$, there exists an $L$-coloring of $G$.

A \emph{linear $L$-coloring} of $G$ is a map $\phi$ defined on $E(G)$ such that $\phi(e)\in L(e)$ for every $e\in E(G)$, and for every color $c \in \bigcup_{e\in E(G)} L(e)$, $\phi^{-1}(c)$ induces a linear forest. The \emph{list linear arboricity} of graph $G$, denoted by $lla(G)$, is the minimum $k$ such that for every list assignment $L=(L(e): e \in E(G))$  with $|L|\ge k$, there exists a linear $L$-coloring of $G$. The list version of the LAC which is called the List Linear Arboricity Conjecture is as follows.

\begin{LLAC}
For every graph $G$ with maximum degree $\D$, $\Lceil\frac{\D}{2}\Rceil \le la(G)=lla(G) \le \Lceil \frac{\D+1}{2}\Rceil $.
\end{LLAC} 
The LLAC was first proposed by An and Wu~\cite{LLAC}, and they proved it holds for planar graphs with maximum degree at least $13$ in anther paper~\cite{list-planar}.
	
In this paper, we prove that 
the LLAC holds asymptotically,  in a manner similar to 
Theorem~\ref{THM:alon}. Indeed our result implies Theorem~\ref{THM:alon}.

In fact, we prove a stronger result, which in order to state, we need the following definitions.
 For a graph $G$ and a list assignment $L=(L(e): e\in E(G))$, let $L(v)=\bigcup_{e\sim v} L(e)$  for every vertex $v\in V(G)$, that is, $L(v)$ is the set of colors which are contained in the list of at least one edge incident with $v$.
For $c\in L(v)$, the \emph{color degree of $v$ with respect to $c$},  denoted by $d_G^L(v,c)$, is the  number of edges $e$ incident with $v$ where $c\in L(e)$.  The \emph{maximum color degree} of $G$ with respect to $L$, which we denote by $\DLG$, is the maximum of $d_G^L(v,c)$ taken over all $v\in V(G)$ and  $c\in L(v)$. 
We are now ready to state our main theorem as follows.

\begin{restatable}{THM}{MAIN}\label{THM:main}
For every $\e >0$, there exists $\de$ such that for every $d > \de$, if $G$ is a graph and $L=(L(e): e\in E(G))$ is a list assignment such that
\begin{itemize}
\item $|L|\ge \frac{d}{2}(1+\epsilon)$, and
\item $\DLG \le d$,
\end{itemize}
then $G$ is linear $L$-colorable.
\end{restatable}

Since $\DLG$ is at most the maximum degree of $G$, Theorem~\ref{THM:main} yields the asymptotic version of the LLAC as a corollary.

\begin{COR}\label{COR:main}
For every $\e >0$, there exists $\De$ such that for every $\D > \De$, if $G$ is a graph with maximum degree $\D$ and $L=(L(e): e\in E(G))$ is a list assignment such that $|L|\ge \frac{\D}{2}(1+\epsilon)$, then $G$ is linear $L$-colorable, and thus $lla(G) \le \frac{\D}{2}(1+\e)$.
\end{COR}

\subsection{Relations to arboricity and edge coloring} 

The \emph{arboricity} of a graph $G$, denoted by $ar(G)$, is the minimum number of forests of $G$ needed to cover all edges of $G$. 
Linear arboricity is a variant of arboricity, and since a linear forest is a forest with maximum degree at most two, this concept can be extended to a covering by forests with bounded maximum degree as follows.
The \emph{$t$-arboricity} of a graph $G$, denoted by $ar_t(G)$, is the minimum number of forests with maximum degree at most $t$  needed to cover all edges of $G$.
Note that $ar_1(G)=\chi'(G)$, where $\chi'(G)$ is the chromatic index of $G$, and $ar_2(G)=la(G)$. Note that $ar_t(G) \ge ar_{t+1}(G)$ for all $t$ and that $ar_{|V(G)|}(t)=ar(G)$.
Given the latter statement, arboricity can be thought of  as $\infty$-arboricity.

For a graph $G$ with maximum degree $\D$, Vizing's Theorem~\cite{VIZING}, also proved by Gupta~\cite{GUPTA}, gives that $ar_1(G) \in \{\D, \D+1\}$ which implies that $ar_1(G)\le \D+1$. 
The LAC also states that 
$ar_2(G) \le \Lceil \frac{\D+1}{2}\Rceil$. 
These works, although the LAC is not proved yet, naturally lead to the question if $ar_t(G)\le \Lceil \frac{\D+1}{t}\Rceil$ for every positive integer $t$, which would be an extension of the LAC. 

But, this question turns out to be false for every $t\ge 3$ since, if $G$ is $\D$-regular, then we have $ar_t(G) \ge \Lceil \frac{\D |V(G)|}{2(|V(G)|-1)}\Rceil \ge \Lceil \frac{\D}{2}\Rceil$ for every $t$  because every spanning tree of $G$ has at most $|V(G)|-1$ edges.
Indeed, we have that $ar_t(G) \ge \max_{H \subseteq G} \Lceil \frac{|E(H)|}{|V(H)|-1} \Rceil$ for every graph $G$ and every $t$ (even for $t=\infty$). Furthermore,  Nash-Williams~\cite{NW} proved that the equality  holds when $t=\infty$, that is, $ar(G)=\max_{H \subseteq G} \Lceil \frac{|E(H)|}{|V(H)|-1} \Rceil$.

The \emph{list $t$-arboricity} 
of $G$, denoted by $lar_t(G)$, 
 is the list version of the $t$-arboricity of $G$  (defined similarly to list linear arboricity). Note that $lar_2(G) = lla(G)$ and that $lar_1(G) = ch'(G)$.
Similarly one can define the \emph{list arboricity} of $G$, $lar(G)$, as the list version of the arboricity of $G$.
The LLAC is asking if $ar_2(G)=lar_2(G)$, and this question can be extended to ask if $ar_t(G)=lar_t(G)$ holds for every positive integer $t$.
Seymour~\cite{PS} showed that it holds when $t=\infty$, that is, $ar(G)=lar(G)$, 
and Theorem~\ref{THM:alon} and~\ref{THM:main} show that $ar_2(G)=lar_2(G)$ holds asymptotically. When $t=1$, this question is equivalent for simple graphs to the well-known List Coloring Conjecture stated below.

\begin{LCC}
For every loopless multigraph $G$, $\chi'(G)=ch'(G)$.
\end{LCC}

The lower bound from Vizing's Theorem and the following theorem by Kahn~\cite{Kahn} confirmed that for simple graphs the list coloring conjecture holds asymptotically.


\begin{KAHN}
For every graph $G$ with maximum degree $\D$, the list edge chromatic number of $G$ is $\D+o(\D)$.
\end{KAHN}

Vizing's theorem plays a crucial role in the proof of Theorem~\ref{THM:alon}, but Kahn's Theorem can not be directly applied to the proof of Theorem~\ref{THM:main}.
Instead, we use a generalization of Kahn's Theorem, (we refer to it as the \emph{color degree version} of Kahn's Theorem) as follows. 

\begin{THM}\label{THM:DK}
For every $\epsilon>0$, there exists $\de$ such that if $d> \de$, then for every graph $G$ and every list assignment $L$ to $E(G)$
such that
\begin{itemize}
\item $|L|\ge (1+\e)d$, and 
\item $\D_G^L \le d$,
\end{itemize}
$G$ is $L$-colorable. 
\end{THM}

Note that Theorem~\ref{THM:DK} implies Kahn's Theorem, but the converse does not hold. Kahn's proof is actually done in terms of $\D_G^L$ though Theorem~\ref{THM:DK} is not stated there; for its statement and other generalizations see~\cite{BDP}.

\subsection{On degree $t$ edge coloring}

We define a \emph{degree $t$ coloring} of $E(G)$ as a coloring of the edges so that every monochromatic subgraph has maximum degree at most $t$. For every positive integer $t$, let $\chi_t'(G)$ denote the minimum number of colors such that $G$ has a degree $t$ coloring using that many colors. Similarly let $ch'_t(G)$ denote the list version of $\chi_t'(G)$.

One might ask questions similar to the LLAC for these two parameters: for a given $t$, does $\chi'_t(G)=ch'_t(G)$ hold? Note that $\chi'_1(G)=\chi'(G)$ and $ch'_1(G)=ch'(G)$. 
Hence the List Coloring Conjecture is equivalent to this question holding in the affirmative for $t=1$.
Surprisingly, the List Coloring Conjecture nearly implies this question holds in the affirmative for every $t$ as follows.
 
It is easy to see that for every graph $G$ with maximum degree $\D$, $\chi'_t(G)\le \Lceil \frac{\chi'(G)}{t} \Rceil$: partition the colors into sets of size $t$ and then merge each set into a new color.
In fact $ch'_t(G)\le \Lceil \frac{ch'(G)}{t} \Rceil$ is also true as the following proposition shows by inverting the merging procedure.
\begin{PROP}\label{PROP:copy}
For every positive integer $t$ and graph $G$, $ch'_t(G) \le \Lceil \frac{ch'(G)}{t} \Rceil$.
\end{PROP}
\begin{proof}
Let $L$ be a list assignment such that $|L|=\Lceil \frac{ch'(G)}{t} \Rceil$. Construct a list assignment $L'=(L'(e)=L(e) \times [t] : e \in E(G))$ from $L$ by copying each color $t$ times. 
Then we have that $|L'| =|L|\times t \ge  ch'(G)$. So there exists an $L'$-coloring $\phi$ of $G$.
We construct a map $\psi$ defined on $E(G)$ from $\phi$ by merging copied colors, that is, if $\phi(e)=(c,i)$ for $c\in L(e)$ and $i \in [t]$, then we set $\psi(e)=c$.  
Evidently, $\psi(e)\in L(e)$, and in $\psi$, every monochromatic subgraph has maximum degree at most $t$. Hence $\psi$ is a degree $t$ $L$-coloring of $G$. Since $L$ was arbitrary, this shows that $ch'_t(G) \le \Lceil \frac{ch'(G)}{t} \Rceil$ as desired.
\end{proof}

Note that $\chi'_t(G) \ge \frac{\chi'(G)}{t+1}$ by applying Vizing's Theorem so as properly color every color class of maximum degree at most $t$ with $t+1$ colors. Hence, 
$$\frac{\chi'(G)}{t+1} \le \chi'_t(G)=ch'_t(G) \le \frac{ch'(G)}{t},$$
for every graph $G$ and every $t$. Thus if the List Coloring Conjecture is true, then $ch'_t(G) \le (1+\frac{1}{t+1})\chi'_t(G)$.
Furthermore, using Kahn's Theorem and Proposition~\ref{PROP:copy}, we have the following.

\begin{THM}\label{THM:Kahncopy}
For every $\e > 0$, there exists $\D_\e$ such that for every positive integer $t$ and graph $G$ with $\D(G) > \D_\e$,
$$ch'_t(G) \le (1+\e)\frac{\D}{t}.$$
\end{THM}

Using Theorem~\ref{THM:DK} and the idea of copying colors as in the proof of Proposition~\ref{PROP:copy}, we deduce the following color degree version of Theorem~\ref{THM:Kahncopy} or equivalently the degree $t$ version of Theorem~\ref{THM:DK}, which will be used to prove one of the main lemmas in Section~\ref{SEC:partition}.
\begin{THM}\label{THM:copy}
For every $\epsilon>0$, there exists $\de$ such that for all $d > \de$ the following holds: for every graph $G$, every positive integer $t$ and every list assignment $L$ of $E(G)$ with $|L|\ge (1+\e)\frac{d}{t}$ and $\D_G^L \le d$, $G$ has a degree $t$ $L$-coloring.
\end{THM}
\begin{proof}
Construct a list assignment $L'=(L'(e)=L(e) \times [t] : e \in E(G))$ from $L$ by copying each color $t$ times. Then we have that $|L'| =|L|\times t \ge  (1+\e)d$. Note that $\D_G^{L'} = \D_G^L \le d$. By Theorem~\ref{THM:DK}, $G$ has an $L$-coloring (assuming the same $\de$). We construct a map $\psi$ defined on $E(G)$ from $\phi$ by merging copied colors, that is, if $\phi(e)=(c,i)$ for $c\in L(e)$ and $i \in [t]$, then we set $\psi(e)=c$.  Evidently, $\psi(e)\in L(e)$, and in $\psi$, every monochromatic subgraph has maximum degree at most $t$. Hence $\psi$ is a degree $t$ $L$-coloring of $G$ as desired.
\end{proof}

\section{Outline and Proof of Theorem~\ref{THM:main}}\label{SEC:proof}
First, a quick word on notation. Let $G$ be a graph and $L$ a list assignment to edges of $G$.
For a subgraph $H$ of $G$, we say $H$ is \emph{linear $L$-colorable} if $H$ is linear $L'$-colorable where $L'=(L(e): e\in E(H))$. We say a map $\phi$ defined on $E(H)$ is a \emph{linear $L$-coloring of $H$} if $\phi$ is a linear $L'$-coloring of $H$.

\subsection{Overview of Alon's proof}

Before proving the main theorem, we take a closer look at the idea of the proof of Theorem~\ref{THM:alon}. We may construe the proof by Alon~\cite{PM} as having three essential steps:

\begin{itemize}
\item[(1)] Finding a degree $2$-coloring of $E(G)$ using $\frac{\D}{2} + o(\D)$ colors such that each monochromatic cycle has length at least $g=\frac{\log \D}{20 \log \log \D}$.
\item[(2)] Finding a subgraph $H$ of maximum degree $o(\D)$ whose edges intersect every monochromatic cycle.
\item[(3)] Finding a degree $2$-coloring of $E(H)$ using $o(\D)$ new colors.
\end{itemize}
To prove (1), one can partition the edges of the graph into about $\frac{\D}{\log^{10} \D}$ subgraphs with maximum degree at most $d=\log^{10} \D + \log^6 \D$ and girth at least $g$ and one remaining part with maximum degree $o(\D)$ with no restriction on the girth. This can be done by repeatedly extracting such high girth subgraphs by choosing edges at random with probability $p=\frac{\log^{10} \D}{\D}$. Each high girth subgraph is then given a degree 2 coloring using $\Lceil \frac{d}{2} \Rceil$ colors while the last part is properly edge colored with $o(\D)$ colors using Vizing's theorem. 

To prove (2), one partitions the colors into sets of size at most $\frac{g}{50e}$ (or $\frac{g}{4}$ if using Haxell~\cite{Haxell}); for each such set, a matching is found that intersects every monochromatic cycle in that set of colors. Such a matching is equivalent to finding an independent transversal of the monochromatic cycles in the line graph, which is possible since each cycle has length at least $g$ and yet each edge has at most $2(\frac{g}{50e})$ incident edges from that set of colors. Then we let $H$ be the union of these $o(\D)$ matchings and hence $H$ has maximum degree $o(\D)$. To prove (3), one can apply Vizing's theorem to properly edge color $H$.

\subsection{Difficulties for list coloring}

What then are the difficulties in transferring Alon's proof to list coloring? For (1), it follows from Theorem~\ref{THM:copy} that $G$ has a degree 2 $L$-coloring yet we cannot guarantee that each monochromatic cycle has length at least $g$. Indeed, we will prove that (1) holds for list coloring however Alon's proof does not carry over because of the important fact that we cannot guarantee that each high girth subgraph has its own unique subset of colors from which to be colored. 

For (2), another difficulty arises in that there could be many more than $\D$ colors and hence we cannot guarantee that the maximum degree of $H$ is $o(\D)$. For (3), while we can use Kahn's Theorem to color the edges of $H$, it may be that the colors used on $E(H)$ are the same as the colors used in step (1), that is, we cannot simply introduce new colors for $E(H)$ since we have to ensure that each edge is colored from its list. Thus another idea is needed there to ensure there are colors that can be used to color $E(H)$.

\subsection{Overcoming the difficulties}

First, we prove that (1) actually holds in the context of list coloring as follows. Let us define a function $q(d) = \frac{\log d}{6 \log \log d}.$

\begin{LEM}\label{LEM:PART}
For $0<\e <1$, if $d$ is sufficiently large, then 
for every graph $G$ and every list assignment $L=(L(e):e\in E(G))$ with $|L| \ge \frac{d}{2} (1+\e)$ and $\DLG\le d$, $G$ has a degree two $L$-coloring such that every monochromatic cycle has length at least $q(d)=\frac{\log d}{6 \log \log d}$.
\end{LEM}

The proof of Lemma~\ref{LEM:PART} can be found in Section~\ref{SEC:partition}. The key idea in the proof of Lemma~\ref{LEM:PART} is that we have every edge retain each color in its list independently with probability $p = \frac{\log^3 d}{d}$. So instead of partitioning the graph into distinct high girth subgraphs and then using separate colors for each subgraph, we extract a high girth subgraph in each color. In expectation, each edge will have a remaining list $L'$ of about $p|L|$ colors. However, the degrees will remain unchanged. This then is where we need the color degree version of Kahn's theorem, actually its degree 2 version as in Theorem~\ref{THM:copy}, to find a degree 2 $L'$-coloring. To apply Theorem~\ref{THM:copy}, we must show that $\D_G^{L'}$ is at most $pd (1+o(1))$. An intermediate lemma (Lemma~\ref{LEM:LEM4}) can be found in Section~\ref{SEC:partition} showing that such an assignment $L'$ with high girth, large lists and small color degree exists.

We will also find a subgraph $H$ of $G$ as in (2) but before describing that let us discuss how to resolve the issue in (3) for coloring $E(H)$. Here we adopt the idea of \emph{reserve colors} from the proof of Kahn's theorem~\cite{Kahn}, choosing a set of colors $R(e) \subseteq L(e)$ to save to use on $E(H)$. To ensure that colors in $R(e)$ can be used on $E(H)$ without causing any conflicts with the colors used in Lemma~\ref{LEM:PART}, 
we choose $R(v)\subseteq L(v)$ for each vertex $v$ and construct two list assignments to edges, 
$$
R=(R(e)=L(e)\cap R(u) \cap R(v):e=uv \in E(G)),
$$
and
$$
L'=(L'(e)=L(e)\setminus (R(u)\cup R(v)): e=uv \in E(G)).
$$
Note that $R(v) \cap L'(v)=\emptyset$ for every vertex $v$. Now we use $L'$ to color $G$ and $R$ to recolor $H$. Thus we actually apply Lemma~\ref{LEM:PART} to $(G,L')$, not $(G,L)$, to obtain a degree $2$ $L'$-coloring of $E(G)$ where every monochromatic cycle has length at least $q(d)$. 

To apply Lemma~\ref{LEM:PART} to $(G,L')$, $|L'|$ must be sufficiently large. In addition, if $|R|$ is large enough, then there will exist a linear $R$-coloring of $H$, in fact a proper $R$-coloring by Theorem~\ref{THM:DK}. The following lemma (whose proof can be found in Section~\ref{SEC:reserve}) shows that we can choose $R(v) \subseteq L(v)$ for every $v\in V(G)$ so that $|L'|$ and $|R|$ are large.

\begin{LEM}\label{LEM:RES}
For every $0<\e <1$, if $d$ is sufficiently large, then 
for every graph $G$ and every list assignment $L=(L(e): e \in E(G))$ with $|L| \ge \frac{d}{2}(1+\e)$ and $\DLG \le d$, 
there exists $R(v)\subseteq L(v)$ for $v\in V(G)$ such that
\begin{itemize}
\item[(1)] $|R| \ge \frac{d}{\log^{1/2} d}(1+\e)$ and
\item[(2)] $|L'| \ge \frac{d}{2} (1+\frac{\e}{2})$.
\end{itemize}
where $R=(R(e)=L(e)\cap R(v)\cap R(u):e=uv \in E(G))$ and $L'=(L'(e)=L(e)-R(u)-R(v): e=uv \in E(G))$.
\end{LEM}

Now let us return to how to find $H$ as in step (2). To color $E(H)$ as in step (3), it would suffice to produce such an $H$ with maximum degree $o(d)$. Here, though, we do not use the results on independent transversals since we cannot split the colors into separate groups. Instead, for each monochromatic cycle from step (1) we pick an edge at random from that cycle to add to $H$. This ensures that the edges of $H$ intersect every monochromatic cycle. We can prove that, with some positive probability, the resulting graph has maximum degree $o(d)$. This would suffice to prove Corollary~\ref{COR:main} as we could then properly edge color $H$ using Kahn's Theorem and the list assignment $R$. However, since we are proving Theorem~\ref{THM:main}, that is the color degree version, we need to prove that $\D_H^R$ is $o(d)$ since we may not be able to control the maximum degree of $G$ let alone $H$. Thus we need the following lemma, whose proof can be found in Section~\ref{SEC:pre}.

\begin{LEM}\label{LEM:Subgraph}
For $0<\e <1$, if $d$ is sufficiently large the following holds: Suppose that $G$ is a graph and $R$ is a list assignment of $E(G)$ such that $|R|\ge \frac{d}{\log^{1/2} d} (1+\e)$ and $\D_G^R\le d$. If $\mathcal{C}$ is a set of edge-disjoint cycles with length at least $q(d)=\frac{\log d}{6\log \log d}$, then there exists a subgraph $H$ of $G$ with $\D_H^R \le \frac{d}{\log^{1/2} d}$ such that $E(H)\cap E(C)\neq \emptyset$ for every $C\in \mathcal{C}$.
\end{LEM}

One slight technical wrinkle arises in the proof of Lemma~\ref{LEM:Subgraph}. If we choose the edges of $H$ at random from the entirety of each monochromatic cycle, then we are unable to control the dependencies needed to guarantee the $\D_H^R$ is small as the cycles may be arbitrarily long. The trick to resolving said wrinkle is to restrict the choice to an arbitrary subset of each monomochromatic cycle of size exactly $q(d)$. In this way, we can ensure the correct dependencies. In fact, we could even apply this to choose an edge from each part of a partition of each monochromatic cycle into paths of length between $q(d)$ and $2q(d)$. Doing that would ensure that the final coloring has no monochromatic path of length more than $4q(d)$ but we did not do this for the reader's sake.

Now to finish the proof, given $H$ as in Lemma~\ref{LEM:Subgraph} along with $R$ as in Lemma~\ref{LEM:RES}, we can find an $R$-coloring of $H$ using Theorem~\ref{THM:DK}; this combined with the linear $L'$-coloring of $G$ from Lemma~\ref{LEM:PART} now yields a linear $L$-coloring of $G$. 

This then is the outline of the proof. For completeness, though, in the remainder of this section, we restate and prove Theorem~\ref{THM:main}, assuming the validity of Lemmas~\ref{LEM:PART},~\ref{LEM:RES} and~\ref{LEM:Subgraph}, whose proofs can be found in Sections~\ref{SEC:partition},~\ref{SEC:reserve} and~\ref{SEC:pre} respectively.

\MAIN*

\begin{proof}
Since $d$ is sufficiently large, and by assumption $|L|\ge \frac{d}{2}(1+\e)$ and $\DLG \le d$, 
Lemma~\ref{LEM:RES} implies that for every vertex $v$, there exists $R(v)\subseteq L(v)$ such that $|R| \ge \frac{d}{\log^{1/2} d}(1+\e)$ and $|L'| \ge \frac{d}{2} (1+\frac{\e}{2})$. 
Recall that for each edge $e=uv$, $R(e)=L(e) \cap R(v)\cap R(u)$ and $L'(e)=L(e)-R(v)-R(u)$, and that for every vertex $v$ of $G$, $R(v)$ and $L'(v)$ are disjoint. Further note that since $R(e)\subseteq L(e)$ for every $e\in E(G)$, it follows that $\D_G^R \le \D_G^L \le d$. Similarly $\D_G^{L'} \le \D_G^L \le d$.

We consider $G$ and $L'$. Note that $|L'|\ge \frac{d}{2}(1+\frac{\e}{2})$ and $\D_G^{L'} \le d$. Thus, by Lemma~\ref{LEM:PART} applied to $(G,L')$ as $d$ is sufficiently large, there exists a degree two $L'$-coloring $\phi$ of $G$ such that every monochromatic cycle has length at least $q$. Let $\mathcal{C}$ be the set of monochromatic cycles in $\phi$. 
Since $\phi$ is a degree $2$ coloring, every pair of cycles in $\C$ is edge-disjoint. 

Since $G$, $R$ and $\C$ satisfy the conditions of Lemma~\ref{LEM:Subgraph}, there exists by Lemma~\ref{LEM:Subgraph} a subgraph $H$ of $G$ such that $\D_H^R \le \frac{d}{\log^{1/2} d}$ and $E(H)\cap E(C)=\emptyset$ for every $C\in \C$. Since $\D_H^R \le \frac{d}{\log^{1/2} d}$ and $|R|\ge \frac{d}{\log^{1/2} d}(1+\e)$, Theorem~\ref{THM:DK} implies that there exists a proper $R$-coloring $\phi'$ of $H$. As $\phi'$ is proper, $\phi'$ is also a linear $R$-coloring of $H$.

Let $\psi$ be the $L$-coloring such that for each edge $e\in E(G)$, $\psi(e)=\phi'(e)$ if $e\in E(H)$ and $\psi(e)=\phi(e)$ otherwise.  Now we claim that $\psi$ is a linear $L$-coloring of $G$. Suppose not. Then there exists a monochromatic cycle $C$ colored $c$. By the definitions of $\phi$ and $\phi'$, $C$ must contain edges $e$ and $f$ sharing a vertex $v$ such that $e \not \in E(H)$ and $f \in E(H)$. Since $e\not \in E(H)$, it follows that $c=\psi(e)=\phi(e) \in L'(e) \subseteq L'(v)$, and since $f \in E(H)$, it follows that $c =\psi(f)=\phi'(f) \in R(f) \subseteq R(v)$. This yields a contradiction since $R(v)$ and $L'(v)$ are disjoint. 
\end{proof}

\section{Probabilistic Preliminaries}\label{SEC:pre}
Proofs of lemmas in this article largely involve the probabilistic method. In this section, we list some theorems regarding probability theory which are used in the proofs.

For a positive integer $n$ and a real number $0\le p \le 1$, we denote by $B(n,p)$ the binomial distribution with $n$ independent variables and probability $p$. Chernoff's Bound shows that every binomial random variable is close to its expected value with high probability, as follows.

\begin{CB}
For $0\le t < np$, 
$$
\PR\left(|B(n,p)-np|>t\right) <2e^{\frac{-t^2}{3np}}.
$$
\end{CB} 

When a random variable is not binomial, but determined by $n$ independent trials, Talagrand's Inequality is useful to show that such a variable is concentrated around its expected value.

\begin{TI}

Let $X$ be a non-negative random variable, not identically 0, which is determined by $n$ independent trials $T_1,\ldots,T_n$, and satisfying the following for some $c,r>0$:
\begin{itemize}
\item changing the outcome of any one trial can affect $X$ by at most $c$, and 
\item for any $s$, if $X\ge s$ then there is a set of at most $rs$ trials whose outcomes certify that $X\ge s$,
\end{itemize}
then for any $0\le t \le E(X)$, 
$$
\PR(|X-E(X)|>t) \le 2e^{-\frac{\beta t^2}{E(X)}}.
$$
for any $\beta < \frac{1}{8c^2 r}$.
\end{TI}
We also need the following two Lemmas: the Lov\'{a}sz Local Lemma and its generalized version~\cite{LL}. Both are used to show that there is a positive possibility that bad events do not occur.
\begin{LLL}
Let $\mathcal{E}$ be a set of events such that for every event $A\in \mathcal{E}$, $\PR(A)\le p <1$ and there exists a set $D_A$ with $|D_A|\le d+1$ such that $A$ is mutually independent of all events in $\mathcal{E}\setminus D_A$. If $4pd \le 1$ then there is a positive probability that none of the events in $\mathcal{E}$ occur.

\end{LLL}

\begin{WLLL}
Let $\mathcal{E}$ be events in a probability space such that for every $A\in \mathcal{E}$, there exists $D_A$ such that $A$ is mutually independent of all events in $\mathcal{E}\setminus D_A$. If there  exists $(x_A \in [0,1): A\in \mathcal{E})$ such that for each $A\in \mathcal{E}$ the following holds: 
$$
\PR(A) \le x_A \prod_{B \in D_A} (1-x_B),
$$
then there is a positive possibility that none of the events in $\mathcal{E}$ occur.
\end{WLLL}

In the remainder of this section, we prove Lemma~\ref{LEM:Subgraph} about the existence of the subgraph $H$ using the Lov\'{a}sz Local Lemma and Talagrand's Inequality.

\begin{proof}[Proof of Lemma~\ref{LEM:Subgraph}]
Recall that $G$ is a graph, $R$ is a list assignment of $E(G)$ with $|R| \ge \frac{d}{\log^{1/2} d} (1+\e)$ and $\D_G^R \le d$, and $\mathcal{C}$ is a set of edge-disjoint cycles with length at least $q = q(d)  = \frac{\log d}{6\log \log d}$. We may assume that for every $e\in E(G)$,  $|R(e)|\le d$ by removing colors from $L(e)$ if necessary. 

For each $C\in \mathcal{C}$, let $S(C)$ be an arbitrary set of $q$ edges of $C$.
Since every  cycle in $\mathcal{C}$ has length at least $q$, such a set $S(C)$ exists. Note that for distinct $C,C'\in \mathcal{C}$, $S(C)$ and $S(C')$ are disjoint since $C$ and $C'$ are edge-disjoint.

For each $C\in \mathcal{C}$, we randomly choose one edge $e_C$ from $S(C)$ and then let $H$ be a subgraph of $G$ such that $E(H) = \bigcup_{C\in \C} e_C$. 
Observe that for every edge $e$ of $G$, the probability that $e$ belongs to $E(H)$ is $\frac{1}{q}$ if there is a cycle $C\in\mathcal{C}$ with $e \in S(C)$ and $0$ otherwise.
So, for $v\in V(H)$ and $c\in R(v)$, the expected value of $d_H^R(v,c)$ is at most $\frac{d_G^R(v,c)}{q}$. So,  
\begin{equation}
\E(d_H^R(v,c))\le \frac{d_G^R(v,c)}{q} \le \frac{\D_G^R}{q} \le \frac{d}{q}=\frac{6d \log \log d}{\log d}.
\end{equation}
Let $A(v,c)$ be the event that $d_H^R(v,c) >\frac{d}{\log^{1/2} d}$ and $\mathcal{E}=\{A(v,c) \mid v \in V(H), c\in R(v)\}$.
We prove by applying the Lov\'{a}sz Local Lemma that with positive probability none of the events in $\mathcal{E}$ occur.

Let $\C(v,c) = \{ C\in \C: \exists e\in S(C)$ such that $v\sim e$ and $c\in R(e)\}$. Let $D_{A(v,c)} = A(v,c)\cup \{ A(v',c'): \C(v',c')\cap \C(v,c)\ne \emptyset\}$. Note that $A(v,c)$ is mutually independent of all events in $\mathcal{E}\setminus D_{A(v,c)}$. Now $|\C(v,c)| \le \D_G^R \le d$. For each $C\in \C$, there are at most $q$ edges in $S(C)$. For each such edge $e$, there are at most $d$ colors in $R(e)$ by assumption. Since each edge has at most two endpoints, it now follows that for each $C\in \C$,  $|\{A(v',c'): C\in \C(v',c')\}|\le 2qd$. Hence $|D_{A(v,c)}|\le 2qd^2+1$. So, it is enough to prove that $\PR(A(v,c)) \le \frac{1}{4(2qd^2+1)}$.

We use Talagrand's inequality to show that $d_H^R(v,c)$ is highly concentrated, where the random trials are for each $C\in \C$ choosing an edge from $S(C)$ to be in $H$.
Observe that changing the outcome of any one trial changes $d_H^R(v,c)$ at most one, and if $d_H^R(v,c)\ge s$ then there is a set of $s$ trials of which outcomes certify $d_H^R(v,c)\ge s$. 
So, by applying Talagrand's inequality with $c=r=1$ and $t= \E\big(d_H^R(v,c)\big)^{\frac12} \cdot \log d$ to $d_H^R(v,c)$, we have
\begin{gather*}
\PR\left(d_H^R(v,c) > \E\big(d_H^R(v,c)\big)+\E\big(d_H^R(v,c)\big)^{\frac12} \cdot \log d\right) \le 2e^{-\frac{\log^2d }{8}}.
\end{gather*}
By (1), we have

$$
\E\big(d_H^R(v,c)\big)+\E\big(d_H^R(v,c)\big)^{\frac12} \cdot \log d
\le \frac{6d\log\log d}{\log d}+\left(\frac{6d\log\log d}{\log d}\right)^{1/2} \cdot \log d \le \frac{d}{\log^{1/2} d},
$$
where the last inequality holds since $d$ is sufficiently large.
Therefore, we have the following:

\begin{align*}
\PR(A(v,c)) &= \PR\left(d_H^R(v,c) >\frac{d}{\log^{1/2} d}\right) \\
&\le \PR\left(d_H^R(v,c) > \E\big(d_H^R(v,c)\big)+\E\big(d_H^R(v,c)\big)^{\frac12}\cdot \log d\right) \le 2e^{-\frac{\log^2 d}{8}}.
\end{align*}
Since $d$ is sufficiently large, it follows that $2e^{-\frac{\log^2 d}{8}} < \frac{1}{4(2qd^2+1)}$, and so $\PR(A(v,c))\le \frac{1}{4(2qd^2+1)}$ as desired.
\end{proof}

\section{Reserving Colors}\label{SEC:reserve}
In this section, we prove Lemma~\ref{LEM:RES} about finding a reserve color assignment $R$ and a resulting list assignment $L'$ that are large enough.

\begin{proof}[Proof of Lemma~\ref{LEM:RES}]
We may assume that for every $e\in E(G)$,  $|L(e)|= \ell=\Lceil \frac{d}{2}(1+\e)\Rceil$ by removing colors from $L(e)$ if necessary. 

For each vertex $v$ and each color $c \in L(v)$, we place $c$ into $R(v)$ with probability $p=\frac{2}{\log^{1/4} d}$. 
For each $e\in E(G)$, let $A_e$ be the event that $|R(e)| < \frac{d}{\log^{1/2} d}(1+\e)$ and let $B_e$ be the event that $|L'(e)| < \frac{d}{2}(1+\frac{\e}{2})$. Let $\mathcal{E}=\{A_e : e \in E(G)\}\cup \{B_e : e\in E(G)\}$. 

We prove that with positive probability none of the events in $\mathcal{E}$ occur. To do this, we apply the \LL. 
For $e=uv \in E(G)$, let 
$$D(e) = \{ f\sim e: L(f)\cap L(e)\ne \emptyset\}.$$ 
Now for each $e\in E(G)$, let 
$$D_{A_e} = D_{B_e} = \{A_f: f\in D(e)\} \cup \{B_f: f\in D(e)\}.$$
Then, $A_e$ and $B_e$ are each mutually independent of all events in $\mathcal{E}\setminus D_{A_e}=\mathcal{E}\setminus D_{B_e}$. Since 
$$
|D(e)| \le \sum_{c\in L(e)} \big(d_G^L(v,c)+d_G^L(u,c)\big) \le 2\ell d={d^2}(1+\e)+2,
$$ it is enough to prove that $\PR(A_e)$ and $\PR(B_e)$ are at most $\frac{1}{8(d^2(1+\e)+2)}$.

Observe that $|R(e)|  \sim B(\ell,p^2)$ and hence the expected value of $|R(e)|$ is $\ell p^2$. Thus, by applying Chernoff's Bound with $t=\E\big(|R(e)|\big)^{\frac12}\cdot \log d$ we have
\begin{equation*}
\PR\left( |R(e)| < \ell p^2 - \ell^{\frac12} \cdot p \cdot \log d\right)<2 e^{-\frac{\log^2{d}}{3}}.
\end{equation*}
Since $d$ is sufficiently large, we have $\frac{d}{\log^{1/2} d}(1+\e)> \big(2d(1+\e)+4\big)^{\frac12} \cdot \log^{3/4}d$. So,
\begin{equation*}
\ell p^2 - \ell^{\frac12} \cdot p \cdot \log d  \ge \frac{2d(1+\e)}{\log^{1/2} d} - \big(2d(1+\e)+4\big)^{\frac12} \cdot \log^{3/4}d
> \frac{d}{\log^{1/2} d}(1+\e).	
\end{equation*}
Thus,
\begin{align*}
 \PR(A_e)&=\PR\left(|R(e)| < \frac{d}{\log^{1/2} d}(1+\e)\right) \le \PR\left( |R(e)| < \ell p^2 - \ell^{\frac12} \cdot p \cdot \log d\right)\\
 &<2 e^{-\frac{\log^2{d}}{3}}< \frac{1}{8(d^2(1+\e)+2)},
\end{align*}
where the last inequality holds since $d$ is sufficiently large.

Similarly $|L'(e)| \sim B(\ell,(1-p)^2)$ and the expected value of $|L'(e)|$ is $\ell (p-1)^2$. By applying Chernoff's Bound with $t=\E\big(|L'(e)|\big)^{\frac12} \cdot \log d$, we have
\begin{equation*}
\PR\left(|L'(e)| < \ell(1-p)^2 - \ell^{\frac12} \cdot (1-p) \cdot \log d\right) <2 e^{-\frac{\log^2{d}}{3}}.
\end{equation*}
Once again, since $d$ is sufficiently large, we have
\begin{align*}
\ell(1-p)^2 - \ell^{\frac12}  \cdot (1-p) \cdot \log d & > \frac{d}{2}\left(1+\frac{\e}{2}\right) +\frac{\e d}{4} -\frac{2d(1+\e)+4}{\log^{1/4} d} -\left(\frac{d}{2}(1+\e)+1\right)^{\frac12}\log d\\
									&> \frac{d}{2}\left(1+\frac{\e}{2}\right),
\end{align*}
and so 
\begin{align*}
\PR(B_e)&=\PR\left(|L'(e)| < \frac{d}{2} \left(1+\frac{\e}{2}\right)\right)\le \PR\left(|L'(e)| < \ell(1-p)^2 - \ell^{\frac12} \cdot (1-p) \cdot \log d\right)\\
&<2 e^{-\frac{\log^2{d}}{3}}< \frac{1}{8(d^2(1+\e)+2)},
\end{align*}
where again the last inequality holds since $d$ is sufficiently large. 
Therefore by the \LL, with positive probability none of the events in $\mathcal{E}$ occur.
\end{proof}

\section{Partitioning $G$ into subgraphs with maximum degree two}\label{SEC:partition}
In this section, we prove Lemma~\ref{LEM:PART} about finding a degree $2$ $L$-coloring such that every monochromatic cycle has length at least $q(d) = \frac{\log d}{6 \log \log d}$. To do that, we use the following lemma.

\begin{LEM}\label{LEM:LEM4}
For every $0<\e<1$, if $d$ is sufficiently large, then for every graph $G$ and every list assignment $L=(L(e): e\in E(G))$ with $|L|\ge \frac{d}{2}(1+\e)$ and $\D_G^L\le d$, there exists a list assignment $L'$ such that
 \begin{itemize}
\item $L'(e) \subseteq L(e)$, 
\item $|L'| \ge (1+\frac{\e}{2}) \frac{\log^3{d}}{2}$, 
\item $\D^{L'}_G\le \log^{3}d +\log^{\frac52} d$, and 
\item for each color $c \in \bigcup_{e\in E(G)}L(e)$,  the subgraph $G_c$ of $G$ induced by $\{e\in E(G)\mid c \in L'(e)\}$ has girth at least $q(d)=\frac{\log d}{6 \log \log d}$.
\end{itemize}
\end{LEM}

Let us prove Lemma~\ref{LEM:PART} assuming Lemma~\ref{LEM:LEM4} before proving Lemma~\ref{LEM:LEM4}.

\begin{proof}[Proof of Lemma~\ref{LEM:PART}]
Since $|L|$ is sufficiently large, there exists a list assignment $L'$ satisfying the conditions in Lemma~\ref{LEM:LEM4}.
In addition, since $d$ is sufficiently large, by Theorem~\ref{THM:copy} there exists a degree $2$ $L'$-coloring $\psi$ of $G$. Since for each color $c\in \bigcup_{e\in E(G)}L(e)$, the subgraph $G_c$ induced by $\{e\in E(G)\mid c \in L'(e)\}$ has girth at least $q(d)=\frac{\log d}{6 \log \log d}$, it follows that every monochromatic cycle in $\psi$ has length at least $q(d)$ as desired.
\end{proof}

Now it remains to prove Lemma~\ref{LEM:LEM4}.
\begin{proof}[Proof of Lemma~\ref{LEM:LEM4}]
We may assume for every edge $e$ that $|L(e)|=\ell = \Lceil \frac{d}{2}(1+\e) \Rceil$ by removing colors from $L(e)$ if necessary.

For an edge $e$ and a color $c\in L(e)$, we place $c$ into $L'(e)$ with probability $p=\frac{\log^3 d}{d}$. Let $\mathcal{C}$ be the set of cycles of length less than $\frac{\log d}{6 \log \log d}$.
For $e\in E(G)$, $v\in V(G)$, $c\in \bigcup_{e\in E(G)}L(e)$ and  $C \in \mathcal{C}$, let
\begin{itemize}
\item $A(e)$ be the event that $|L'(e)| < \left(1+\frac{\e}{2}\right)\frac{\log^{3}d}{2}$, and
\item $B(v,c)$ the event that $d_G^{L'}(v,c)>\log^{3}d +\log^{\frac52}d$, and
\item $D(C,c)$ the event that  for every edge $f$ of $C$, $L'(f)$ contains $c$.
\end{itemize}
Let 
$$
\mathcal{E}=\{A(e): e\in E(G)\} \cup \{B(v,c): v\in V(G), c\in L'(v)\} \cup \{D(C,c): c\in \bigcup_{f\in E(G)}L(f), C \in \mathcal{C}\}.
$$
We prove that with positive probability none of the events in $\mathcal{E}$ occur. To show this, we use the general version of the \LL.

Observe that $|L'(e)| \sim B(\ell,p)$. Thus the expected value of $|L'(e)|$ is $p\ell$. By applying Chernoff's bound with $t=(p\ell)^{\frac12}\cdot\log d$ to $|L'(e)|$, we obtain

\begin{eqnarray}
\PR\left(|L'(e)| < p\ell - (p\ell)^{\frac12}\cdot\log d\right) &< 2e^{-\frac{\log^2 d}{3}}\label{2}.
\end{eqnarray}
Yet
\begin{align*}
 p\ell - (p\ell)^{\frac12}\cdot\log d&\ge \frac{\log^{3}d}{2}\left(1+\frac{\e}{2}\right) +\frac{\e\log^{3}d}{4} -\left(\frac{1+\e}{2}+\frac1d\right)^{1/2}\cdot \log^{\frac52}d\\
 							&>\frac{\log^{3}d}{2}\left(1+\frac{\e}{2}\right).
 							\end{align*}
Combining this with (\ref{2}), we have 
\begin{align*}
\PR(A(e))&=\PR\left(|L'(e)| <\left(1+\frac{\e}{2}\right) \frac{\log^3 d}{2}\right) \\
&\le \PR\left(|L'(e)| < p\ell - (p\ell)^{1/2}\cdot \log d\right) < 2e^{-\frac{\log^2 d}{3}}.
\end{align*}

Similarly $d^{L'}_G(v,c)\sim B(d^{L}_G (v,c),p)$ and hence the expected value of $d^{L'}_G(v,c)$ is $pd^{L}_G (v,c)$. By applying Chernoff's bound with $t=\big(p d_G^{L}(v,c)\big)^{\frac12}\cdot \log d$ to $d^{L'}_G(v,c)$, we obtain 
\begin{eqnarray}
\PR\left(d_G^{L'}(v,c) > pd_G^L(v,c)+\big(p d_G^{L}(v,c)\big)^{\frac12}\cdot \log d\right) &< 2e^{-\frac{\log^2 d}{3}}\label{3}.
\end{eqnarray}
Yet 
$$pd_G^L(v,c)+\big(p d_G^{L}(v,c)\big)^{\frac12}\cdot \log d \le \log^3 d+\log^{\frac52} d.$$
Combining this with (\ref{3}), we have
\begin{align*}
\PR(B(v,c))&=\PR\left(d_{G}^{L'} (v,c) >\log^3 d +\log^{\frac52}d\right) \\
&<  2e^{-\frac{\log^2 d}{3}}.
\end{align*}
We also easily obtain that 
$
\PR(D(C,c)) \le p^{|C|},
$
since the events that $(c\in L'(f): f\in E(C))$ are independent and occur with probability $p$.

Let $q = q(d)=\frac{\log d}{6\log \log d}$.
To apply the general version of the \LL, we define 
\begin{itemize}
\item $x_{A_e}= \frac{1}{d^q}$, 
\item $x_{B(v,c)}=\frac{1}{d^q}$, 
\item $x_{D(C,c)}=\frac{1}{d^{|C|-1}}$.
\end{itemize}
Furthermore let
\begin{itemize}
\item $D_{A_e} = \{B(v,c): v\sim e, c\in L(e)\} \cup \{D(C,c): c\in L(e), e\in E(C)\}$,
\item $D_{B(v,c)} = \{A_e: e\sim v, c\in L(e)\} \cup \{B(u,c): u\in N(v)\} \cup \{D(C,c): v\in V(C)\}$,
\item $D_{D(C,c)} = \{A_e: e\in E(C), c\in L(e)\} \cup \{B(v,c): v\in V(C)\} \cup \{D(C',c): E(C)\cap E(C')\ne \emptyset\}$.
\end{itemize}
Note that for each $F\in \mathcal{E}$, $F$ is mutually independent of all events in $\mathcal{E}\setminus D_{F}$ where $D_{F}$ is defined as above. Therefore, with positive probability none of the events in $\mathcal{E}$ occur provided that the three following inequalities hold:
\begin{gather}
\PR(A_e)<2e^{-\frac{\log^2 d}{3}
}< \frac{1}{d^q} \cdot \left(1-\frac{1}{d^q}\right)^{d(1+\e)+2 } \cdot \prod_{r=3}^{q} \left(1-\frac{1}{d^{r-1}}\right)^{\ell d^{r-2}}\label{4}\\
\PR(B(v,c))<2e^{-\frac{\log^2 d}{3}}< \frac{1}{d^q}\cdot \left(1-\frac{1}{d^q}\right)^{2d} \cdot \prod_{r=3}^q \left(1-\frac{1}{d^{r-1}}\right)^{d^{r-1}}\label{5}\\
\PR(D(C,c))\le    \left(\frac{\log^3 d}{d}\right)^{|C|} <\frac{1}{d^{|C|-1}} \cdot \left(1-\frac{1}{d^q}\right)^{2|C|} \cdot \prod_{r=3}^q 
\left(1-\frac{1}{d^{r-1}}\right)^{|C|d^{r-2}}\label{6}
\end{gather}

Since $\e <1$ and $d$ is sufficiently large, we have that $\ell < d$ and it now follows that (\ref{5}) implies (\ref{4}). To prove (\ref{5}) and (\ref{6}), we use the following two inequalities.
\begin{itemize}
\item for positive integers $a$ and $b$, $\left(1-\frac{1}{a}\right)^b \ge 1- \frac{b}{a}$
\item for sufficiently large positive integers $n$, $\left(1-\frac{1}{n} \right)^n > \frac{1}{3}$
\end{itemize}
Since $d$ is sufficiently large, we have

\begin{align*}
\frac{1}{d^q} \cdot \left(1-\frac{1}{d^q}\right)^{2d } \cdot \prod_{r=3}^{q} \left(1-\frac{1}{d^{r-1}}\right)^{d^{r-1}}&
> \frac{1}{d^q} \cdot \left(1-\frac{2}{d^{q-1}}\right) \cdot \prod_{r=3}^{q} \frac13\\
&>\frac{1}{d^q} \cdot \frac{1}{d} \cdot \frac{1}{d^{q}}
=d^{-\frac{\log d}{3\log\log d}-1}\\
&>d^{-\frac{\log d}{3}+1}
>2 e^{-\frac{\log^2 d}{3}}.
\end{align*}
So, (\ref{5}) holds.

To prove that (\ref{6}) holds, let $k=|C|$. Since $k\le q$ and $d$ is sufficiently large, we have
\begin{align*}
\frac{1}{d^{k-1}} \cdot \left(1-\frac{1}{d^q}\right)^{2k} \cdot \prod_{r=3}^q 
\left(1-\frac{1}{d^{r-1}}\right)^{kd^{r-2}}
&>\frac{1}{d^{k-1}} \cdot \left( 1-\frac{2k}{d^q}\right) \cdot \prod_{r=3}^q \left(1-\frac{k}{d}\right)\\
&>\frac{1}{d^{k-1}} \cdot \left(1-\frac{2q}{d^q}\right) \cdot \left(1-\frac{q}{d}\right)^q
\\ 
&>\frac{1}{d^{k-1}}\cdot \left(1-\frac{2q}{d}\right) \cdot \left(1-\frac{q^2}{d}\right)
\\
&> \frac{1}{d^{k-1}}\cdot \frac{1}{d^{\frac12}}=\frac{1}{d^{k-\frac12}} \\
&\ge \frac{1}{d^{k-\frac{k}{2q}}}=\left(\frac{\log^3 d}{d}\right)^k,
\end{align*}
and we obtain the inequality (\ref{6}), where the last equality holds since
$$
\log^3 d =d^{\frac{3\log \log d}{\log d}}.
$$
Therefore, by the general version of the \LL, with positive probability none of the events in $\mathcal{E}$ occur, and so there exists a list assignment $L'$ satisfying the conditions in the statement.
\end{proof}

\bibliographystyle{plain}
\bibliography{ref}

\end{document}